\documentclass[preprint,12pt]{elsarticle}

\usepackage{amscd}
\usepackage{amsmath}
\usepackage{amsthm}
\usepackage{amsfonts}
\usepackage{graphicx}
\usepackage{amssymb}
\usepackage{color}
\usepackage{amsbsy}
\usepackage{graphicx}
\usepackage{amssymb,color,amsbsy}
\usepackage{mathtools}
\usepackage{leftindex}

\usepackage{float}
\usepackage{comment}

\usepackage{stmaryrd}

\usepackage[ruled]{algorithm2e}

\usepackage[matrix,arrow]{xy}

\DeclareMathAlphabet{\mathpzc}{OT1}{pzc}{m}{it}

\usepackage{pgfpages}
\usepackage{tikz}
\usetikzlibrary{shapes.geometric}
\usetikzlibrary{decorations.pathmorphing}
\usetikzlibrary{arrows}
\usetikzlibrary{backgrounds}
\usetikzlibrary{positioning}
\usetikzlibrary{fit}
\usepackage{caption}
\usepackage{amsthm}

\usepackage{amsmath}

\usepackage[pagebackref=true, bookmarksopen=true,colorlinks=true, linkcolor=red,citecolor=blue]{hyperref}

\usepackage[capitalise]{cleveref}  

\global\long\def\ii{\cap}%

\global\long\def\u{\cup}%

\global\long\def\I{\bigcap}%

\global\long\def\U{\bigcup}%

\global\long\def\s{\subset}%

\global\long\def\ñ{\sim}%

\global\long\def\core#1{\textnormal{core}(#1)}%

\global\long\def\critical#1{\textnormal{CritIndep}(#1)}%

\global\long\def\mcritical#1{\textnormal{MaxCritIndep}(#1)}%

\global\long\def\ker#1{\textnormal{ker}(#1)}%

\global\long\def\nucleus#1{\textnormal{nucleus}(#1)}%

\global\long\def\diadem#1{\textnormal{diadem}(#1)}%

\global\long\def\corona#1{\textnormal{corona}(#1)}%

\global\long\def\a#1{\left|#1\right|}%

\usepackage{tikz-cd}

\usepackage{tkz-graph}

\usepackage{multicol}

\usepackage{subcaption}

\newtheorem{theorem}{Theorem}[section]

\newtheorem{corollary}[theorem]{Corollary}

\newtheorem{lemma}[theorem]{Lemma}

\newtheorem{problem}[theorem]{Problem}

\newtheorem{observation}[theorem]{Observation}

\usepackage{comment}
\begin{document}

	\begin{abstract}
		Let $G$ be a finite simple graph.  An independent set $I$ of $G$ is critical if $\left|I\right|-\left|N(I)\right|\ge\left|J\right|-\left|N(J)\right|$
		for every independent set $J$ of $G$. A critical independent set is maximum if it has maximum cardinality. 	
		The \emph{core} and the \emph{nucleus} of $G$ are defined as the intersection of all maximum independent sets and the intersection of all maximum critical independent sets, respectively. In 2019, Jarden, Levit and Mandrescu posed the problem of characterizing the graphs satisfying $\core G=\nucleus G$. In this paper we provide a complete solution to this problem. Using Larson’s independence decomposition, which partitions any graph into a König--Egerváry component $L_G$ and a $2$-bicritical component $L_G^c$, we establish that $\core G=\nucleus G$ holds if and only if $\core {L_G^c}=\emptyset$ and no vertex of $\corona G$ lies in the boundary between $L_G$ and $L_G^c$. We also show that the same boundary condition is equivalent to the identity $\diadem G=\corona G\ii L(G)$. Several consequences and related structural properties are also derived.
	\end{abstract}

\begin{keyword} 	Maximum Critical Independent Set, Diadem, König--Egerváry
	graph, Corona, Nucleus, 2-bicritical	\MSC 05C70, 05C75 \end{keyword}
\begin{frontmatter} 
	\title{Graphs with $\core G=\nucleus G$}

		\author[LEVIT]{Vadim E. Levit} 	\ead{levitv@ariel.ac.il}
	
	\author[MANDRESCU]{Eugen Mandrescu} 	\ead{eugen_m@hit.ac.il}
	
	\author[IMASL,DEPTO]{Kevin Pereyra} 	\ead{kdpereyra@unsl.edu.ar}


	\address[LEVIT]{Department of Mathematics, Ariel University, Ariel, Israel.}

	\address[MANDRESCU]{Department of Computer Science, Holon Institute of Technology, Holon, Israel.}

	\address[IMASL]{Instituto de Matem\'atica Aplicada San Luis, Universidad Nacional de San Luis and CONICET, San Luis, Argentina.}
	\address[DEPTO]{Departamento de Matem\'atica, Universidad Nacional de San Luis, San Luis, Argentina.} 	
	
	\date{Received: date / Accepted: date} 
	
\end{frontmatter} %

\section{Introduction}\label{sss0}

Let $\alpha(G)$ denote the cardinality of a maximum independent set,
and let $\mu(G)$ be the size of a maximum matching in $G=(V,E)$.
It is known that $\alpha(G)+\mu(G)$ equals the order of $G$,
in which case $G$ is a König--Egerváry graph 
\cite{deming1979independence,gavril1977testing,sterboul1979characterization}.
König--Egerváry graphs have been extensively studied
\cite{bourjolly2009node,jarden2017two,jaume5404413kr,levit2006alpha,levit2012critical}.
It is known that every bipartite graph is a König--Egerváry graph; this follows from classical results of K{\H{o}}nig and Egerv{\'a}ry \cite{egervary1931combinatorial,konig1915}. These graphs were independently introduced by Deming \cite{deming1979independence}, Sterboul \cite{sterboul1979characterization}, and Gavril \cite{gavril1977testing}.

A graph G is considered $2$-bicritical if, after removing any two distinct vertices, the remaining subgraph has a perfect matching.
The notion of $2$-bicritical graphs was introduced
in \cite{pulleyblank1979minimum}, and they can be characterized as follows.
\begin{theorem}[\cite{pulleyblank1979minimum}\label{1928u3123}\label{lmaslmaslmlmlm}]
	A graph $G$ is $2$-bicritical if and only if $\left|N(S)\right|>\left|S\right|$
	for every nonempty independent set $S\subseteq V(G)$.
\end{theorem}
The class of $2$-bicritical graphs can be regarded as the structural counterpart of König--Egerváry graphs \cite{larson2011critical,kevinSDKEGE,kevinPOSYFACTOR,kevinSDKECHAR}. It is important to note that \cite{pulleyblank1979minimum} shows that almost every graph is a $2$-bicritical graph.

The main tool used in this work is Larson's independence decomposition \cite{larson2011critical}, which partitions a graph into two parts: one that induces a König--Egerváry graph $L_G$, and another that induces a $2$-bicritical graph $L_G^{c}$.

Let $\Omega^{*}(G)=\left\{ S:S\textnormal{ is an independent set of }G\right\}$,
$\Omega(G)=\{S:S$ is a maximum independent set of $G\}$,
$\textnormal{core}(G)=\I\left\{ S:S\in\Omega(G)\right\}$ 
\cite{levit2003alpha+}, and 
$\textnormal{corona}(G)=\U\left\{ S:S\in\Omega(G)\right\}$ 
\cite{boros2002number}. 
The number $d_{G}(X)=\left|X\right|-\left|N(X)\right|$ is the 
difference of the set $X\s V(G)$, and 
$d(G)=\max\{d_{G}(X):X\s V(G)\}$ is called the \emph{critical difference} of $G$.
A set $U\s V(G)$ is \emph{critical} if $d_{G}(U)=d(G)$ 
\cite{zhang1990finding}. 
The number $d_{I}(G)=\max\left\{ d_{G}(X):X\in\Omega^{*}(G)\right\}$ 
is called the \emph{critical independence difference} of $G$. 
If a set $X\s\Omega^{*}(G)$ satisfies $d_{G}(X)=d_{I}(G)$, then it is called 
a \emph{critical independent set} \cite{zhang1990finding}. 
Clearly, $d(G)\ge d_{I}(G)$ holds for every graph. 
It is known that $d(G)=d_{I}(G)$ for all graphs 
\cite{zhang1990finding}. 
We define 
$\critical{G}=\{ S:S$ is a critical independent set of $G\}$ and \\ 
$\mcritical{G}=\{S:S$ is a maximum critical independent 
set of $G\}$. Recall the following:
$\textnormal{ker}(G)=\I\critical{G}$ 
\cite{levit2012vertices,lorentzen1966notes,schrijver2003combinatorial},
$\textnormal{nucleus}(G)=\I\mcritical{G}$ \cite{jarden2019monotonic}, and 
$\textnormal{diadem}(G)=\U\critical{G}$ 
\cite{short2015some}.  Actually, every critical independent set is contained in a
maximum critical independent set, and a maximum critical independent set can be found in polynomial time \cite{larson2007note}.  Also note that $\diadem{G}=\U\critical{G}=\U\mcritical{G}$.

\medskip

Several interesting phenomena occur when $\core G$ is a critical independent set, and these have been studied in \cite{jarden2019monotonic}. In \cite{lmkcorecritical}, the graphs for which $\core G$ is a critical independent set are completely characterized.

\begin{theorem}[\label{iasjdpoaisjdpio}\cite{lmkcorecritical}]
	For every graph, $\core G$ is a critical independent set if and only if $\core{L_{G}^{c}}=\emptyset$.
\end{theorem}

However, the condition in \cref{iasjdpoaisjdpio} does not ensure the equality $\core G=\nucleus G$, as demonstrated by the example in \cref{adsasdasd123}.

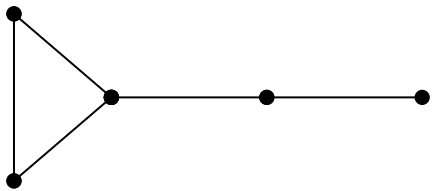
\begin{figure}[H]
	\begin{center}

		\tikzset{every picture/.style={line width=0.75pt}} 
		
		\begin{tikzpicture}[x=0.75pt,y=0.75pt,yscale=-1,xscale=1]
			
			\draw    (107.5,210.03) -- (107.5,125.73) ;
			\draw [shift={(107.5,125.73)}, rotate = 270] [color={rgb, 255:red, 0; green, 0; blue, 0 }  ][fill={rgb, 255:red, 0; green, 0; blue, 0 }  ][line width=0.75]      (0, 0) circle [x radius= 3.35, y radius= 3.35]   ;
			\draw [shift={(107.5,210.03)}, rotate = 270] [color={rgb, 255:red, 0; green, 0; blue, 0 }  ][fill={rgb, 255:red, 0; green, 0; blue, 0 }  ][line width=0.75]      (0, 0) circle [x radius= 3.35, y radius= 3.35]   ;
			\draw    (107.5,125.73) -- (156.63,167.88) ;
			\draw [shift={(156.63,167.88)}, rotate = 40.63] [color={rgb, 255:red, 0; green, 0; blue, 0 }  ][fill={rgb, 255:red, 0; green, 0; blue, 0 }  ][line width=0.75]      (0, 0) circle [x radius= 3.35, y radius= 3.35]   ;
			\draw [shift={(107.5,125.73)}, rotate = 40.63] [color={rgb, 255:red, 0; green, 0; blue, 0 }  ][fill={rgb, 255:red, 0; green, 0; blue, 0 }  ][line width=0.75]      (0, 0) circle [x radius= 3.35, y radius= 3.35]   ;
			\draw    (156.63,167.88) -- (107.5,210.03) ;
			\draw [shift={(107.5,210.03)}, rotate = 139.37] [color={rgb, 255:red, 0; green, 0; blue, 0 }  ][fill={rgb, 255:red, 0; green, 0; blue, 0 }  ][line width=0.75]      (0, 0) circle [x radius= 3.35, y radius= 3.35]   ;
			\draw [shift={(156.63,167.88)}, rotate = 139.37] [color={rgb, 255:red, 0; green, 0; blue, 0 }  ][fill={rgb, 255:red, 0; green, 0; blue, 0 }  ][line width=0.75]      (0, 0) circle [x radius= 3.35, y radius= 3.35]   ;
			\draw    (235,167.88) -- (156.63,167.88) ;
			\draw [shift={(156.63,167.88)}, rotate = 180] [color={rgb, 255:red, 0; green, 0; blue, 0 }  ][fill={rgb, 255:red, 0; green, 0; blue, 0 }  ][line width=0.75]      (0, 0) circle [x radius= 3.35, y radius= 3.35]   ;
			\draw [shift={(235,167.88)}, rotate = 180] [color={rgb, 255:red, 0; green, 0; blue, 0 }  ][fill={rgb, 255:red, 0; green, 0; blue, 0 }  ][line width=0.75]      (0, 0) circle [x radius= 3.35, y radius= 3.35]   ;
			\draw    (156.63,167.88) ;
			\draw [shift={(156.63,167.88)}, rotate = 0] [color={rgb, 255:red, 0; green, 0; blue, 0 }  ][fill={rgb, 255:red, 0; green, 0; blue, 0 }  ][line width=0.75]      (0, 0) circle [x radius= 3.35, y radius= 3.35]   ;
			\draw [shift={(156.63,167.88)}, rotate = 0] [color={rgb, 255:red, 0; green, 0; blue, 0 }  ][fill={rgb, 255:red, 0; green, 0; blue, 0 }  ][line width=0.75]      (0, 0) circle [x radius= 3.35, y radius= 3.35]   ;
			\draw    (313.37,167.88) -- (235,167.88) ;
			\draw [shift={(235,167.88)}, rotate = 180] [color={rgb, 255:red, 0; green, 0; blue, 0 }  ][fill={rgb, 255:red, 0; green, 0; blue, 0 }  ][line width=0.75]      (0, 0) circle [x radius= 3.35, y radius= 3.35]   ;
			\draw [shift={(313.37,167.88)}, rotate = 180] [color={rgb, 255:red, 0; green, 0; blue, 0 }  ][fill={rgb, 255:red, 0; green, 0; blue, 0 }  ][line width=0.75]      (0, 0) circle [x radius= 3.35, y radius= 3.35]   ;

		\end{tikzpicture}

	\end{center}
	\caption{A graph satisfying $\core {L^c_G}=\emptyset$ while $\core G\neq\nucleus G$}.
	\label{adsasdasd123}
\end{figure}

When $\core G$ is a critical independent set, it is known that $\core G \subseteq \nucleus G$ \cite{jarden2019monotonic}, and equality is valid when $\diadem G=\corona G$. Consequently, in \cite{jarden2019monotonic} the following problem is posed: 

\begin{problem}[\label{asdpoasok}\cite{jarden2019monotonic}\label{ioasdoiasi12}]
	Characterize the graphs enjoying $\core{G}=\nucleus{G}$.
\end{problem}

It is known that every König--Egerváry graph satisfies $\core G=\nucleus G$ \cite{levit2012critical}. In this work we solve \cref{asdpoasok} by showing that, in addition to the condition $\core{L_{G}^{c}}=\emptyset$ from \cref{iasjdpoaisjdpio}, the only extra obstruction is the presence of vertices of $\corona G$ on the boundary between the two sides of Larson's partition. Our approach also incorporates the diadem: we prove that $\corona G\ii\partial_{L}(G)=\emptyset$ is equivalent to $\diadem G=\corona G\ii L(G)$.

 The paper is organized as follows. 
 In \cref{sss0} we present the general context of the problem and introduce the main concepts and definitions. 
 In \cref{sss1} we fix the notation that will be used throughout the paper. 
 In \cref{sss2} we establish the characterization of graphs with $\core G=\nucleus{G}$, together with companion results describing the corresponding diadem, thereby solving \cref{asdpoasok}.  
 Finally, \cref{open} discusses concluding remarks and directions for further research.

\section{Preliminaries}\label{sss1}
All graphs considered in this paper are finite, undirected, and simple. 
For any undefined terminology or notation, we refer the reader to 
Lovász and Plummer \cite{LP} or Diestel \cite{Distel}.

Let \( G = (V, E) \) be a simple graph, where \( V = V(G) \) is the finite set of vertices and \( E = E(G) \) is the set of edges.  A subgraph of \( G \) is a graph \( H \) such that \( V(H) \subseteq V(G) \) and \( E(H) \subseteq E(G) \). A subgraph \( H \) of \( G \) is called a \textit{spanning} subgraph if \( V(H) = V(G) \). For two v sets $X,Y\subseteq V(G)$, we denote by $E(X,Y)$ the set of
edges $uv\in E(G)$ such that $u\in X$ and $v\in Y$.

Let \( e \in E(G) \) and \( v \in V(G) \). We define \( G - e = (V, E - \{e\}) \) and \( G - v = (V - \{v\}, \{uw \in E : u,w \neq v\}) \). If \( X \subseteq V(G) \), the \textit{induced} subgraph of \( G \) by \( X \) is the subgraph \( G[X]=(X,F) \), where \( F=\{uv \!\in\! E(G) : u, v \!\in \! X\} \). The union of two graphs $G$ and $H$ is the graph $G\cup H$
with $V(G\cup H)=V(G)\cup V(H)$ and $E(G\cup H)=E(G)\cup E(H)$.

The number of vertices in a graph $G$ is called the \textit{order} of the graph and is denoted $n(G)$.
A \textit{cycle} in $G$ is called \textit{odd} (resp. \textit{even}) if it has an odd (resp. even) number of edges.

For a vertex $v\in V(G)$, the \emph{neighborhood} of $v$ is
\[
N_G(v)=\{u\in V(G): uv\in E(G)\}.
\]
When no confusion arises, we write $N(v)$ instead of $N_G(v)$. For a set $S\subseteq V(G)$, the \emph{neighborhood} of $S$ is
\[
N_G(S)=\bigcup_{v\in S} N_G(v).
\]

A \textit{matching} \(M\) in a graph \(G\) is a set of pairwise non-adjacent edges. The \textit{matching number} of \(G\), denoted by  \(\mu(G)\), is the maximum cardinality of any matching in \(G\). Matchings induce an involution on the vertex set of the graph: \(M:V(G)\rightarrow V(G)\), where \(M(v)=u\) if \(uv \in M\), and \(M(v)=v\) otherwise. If \(S, U \subseteq V(G)\) with \(S \cap U = \emptyset\), we say that \(M\) is a matching from \(S\) to \(U\) if \(M(S) \subseteq U\). A matching $M$ is \emph{perfect} if $M(v)\neq v$ for every vertex
of the graph. A matching is \emph{near-perfect} if \( \left|{v \in V(G) : M(v) = v}\right| = 1 \).  A graph is a factor-critical graph if $G-v$ has perfect matching
for every vertex $v\in V(G)$. The \emph{deficiency} of a graph $G$, denoted by $\textnormal{def}(G)$,
is defined as $\textnormal{def}(G)=\a G-2\mu(G)$.

A vertex set \( S \subseteq V \) is \textit{independent} if, for every pair of vertices \( u, v \in S \), we have \( uv \notin E \). 
The number of vertices in a maximum independent set is denoted by \( \alpha(G) \).

The Gallai–Edmonds decomposition will play an important role in this work.

\begin{theorem}
	[\label{ge}\cite{edmonds1965paths,gallai1964maximale} Gallai--Edmonds
	structure theorem]
	Let $G$ be a graph, and define
	\begin{align*}
		D(G) & = \{ v : \textnormal{there exists a maximum matching that misses } v \}, \\
		A(G) & = \{ v : v \textnormal{ is adjacent to some } u \in D(G), \textnormal{ but } v \notin D(G) \}, \\
		C(G) & = V(G) - (D(G) \cup A(G)).
	\end{align*}
	
	\noindent If $G_{1}, \dots, G_{k}$ are the connected components of $G[D(G)]$
	and $M$ is a maximum matching of $G$, then:
	\begin{enumerate}
		\item $M$ covers $C(G)$ and matches $A(G)$ into distinct components of $G[D(G)]$.
		\item Each $G_i$ is a factor-critical graph, and the restriction of $M$ to $G_i$ 
		is a near-perfect matching.
		\item Each nonempty $S\subseteq A(G)$ is adjacent to at least $|S|+1$ components of $G[D(G)]$.
	\end{enumerate}
\end{theorem}

\section{Main Results}\label{sss2}

In this section we establish the main result of the paper. We first prove several preliminary lemmas, and then obtain a characterization of the graphs whose core equals the nucleus.
In \cite{larson2011critical}, Larson introduces the following decomposition theorem.

\begin{theorem}[\cite{larson2011critical}\label{larsonthm}]
	For any graph $G$, there is a unique set $L(G)\s V(G)$
	such that
	\begin{enumerate}
		\item $\alpha(G)=\alpha(G[L])+\alpha(G[V(G)-L(G)])$,
		\item $G[L(G)]$ is a König-Egerváry graph,
		\item for every non-empty independent set $I$ in $G[V(G)-L(G)]$, we have $\left|N(I)\right|>\left|I\right|,$
		and
		\item for every maximum critical independent set $J$ of $G$,  $L(G)=J\u N(J)$.
	\end{enumerate}
\end{theorem}

Throughout the remainder of the paper, $L(G)$ and $L^{c}(G)=V(G)-L(G)$ denote the sets
of \cref{larsonthm}; moreover, to simplify the notation, we define the induced graphs
\begin{align*}
	L_{G} & =G[L(G)],\\
	L_{G}^{c} & =G[L^{c}(G)].
\end{align*}

\begin{observation}
	By \cref{larsonthm}, for every graph $G$ with $L^c(G)\neq \emptyset$, it follows that $L_{G}^{c}$ is a $2$-bicritical graph.
\end{observation}

\begin{lemma}
	[\cite{larson2007note}]\label{matchinglemma}
	Let $I$ be a critical independent set of $G$. Then there exists a maximum matching
	of $G$ that matches $N(I)$ into $I$.
\end{lemma}

We now recall a structural description of $\core G$ in terms of Larson's partition.
This result allows us to separate the contribution of the two sides $L_G$ and $L_G^c$.

\begin{lemma}[\cite{lmkcorecritical}\label{asdklasdookokasdklasdookokasdklasdookok}\label{asdiuj12oi3j}]
	For every graph
	\[
	\core G=\core{L_{G}^{c}}\u(\core G \ii L(G)).
	\]
	Moreover,
	\[
	\core G \ii L(G)=\I_{S\in\Omega(G)}\left(S\ii L(G)\right)\subseteq \nucleus{G}.
	\]
\end{lemma}

\begin{theorem}[\cite{KEVINcoreker}\label{asdklasdookokasdklasdookok}\label{asdokasdpokas}]
	The equality $\ker G=D(L_{G})$ holds for every graph $G$.
\end{theorem}

At this point, the main difficulty becomes apparent.
While $\core G$ and $\mcritical G$ are largely controlled by the structure of $L_G$, vertices that lie on the interface between $L_G$ and $L^c_G$ may interfere with this behavior.
To capture this phenomenon, we introduce the boundary set
\[
\partial_{L}(G)=\left\{ v\in L(G):N(v)\ii L^{c}(G)\neq\emptyset\right\}.
\]
\noindent By \cref{larsonthm}, it follows that $\partial_{L}(G)\subseteq N(I)$
for every $I\in\mcritical G.$

\begin{lemma}\label{asdklasdookok}
	If $I$ is a maximum critical independent set of $G$, then
	\begin{align*}
		\core G\ii L(G) & =M\left(N(I)-\corona G\right)\u D(L_{G}),
	\end{align*}
	\noindent for every maximum matching $M$ of $L_{G}$. 
\end{lemma}

\begin{proof}
	Let $I\in\mcritical G$ and let $M$ be a maximum matching of $L_{G}$. By \cref{larsonthm},
	$L(G)=I\u N(I)$ and $L_{G}$ is a König--Egerváry graph.
	By \cref{matchinglemma}, there exists a maximum matching of $G$ that matches $N(I)$ into $I$.
	Its restriction to $L_{G}$ is a matching of cardinality $\a{N(I)}$, and hence
	$\mu(L_{G})\ge\a{N(I)}$. Since $I$ is an independent set of $L_{G}$ and
	$\alpha(L_{G})+\mu(L_{G})=\a{L(G)}=\a I+\a{N(I)}$, we obtain
	\[
	\a I\le \alpha(L_{G})=\a{L(G)}-\mu(L_{G})\le \a I+\a{N(I)}-\a{N(I)}=\a I.
	\]
	Therefore, $I\in\Omega(L_{G})$ and $\mu(L_{G})=\a{N(I)}$.
	Consequently, every maximum matching of $L_{G}$ has cardinality $\a{N(I)}$.
	Since $I$ is independent, no edge of $M$ lies inside $I$.
	If some edge of $M$ had both endpoints in $N(I)$, then $M$ would cover at least
	$\a{N(I)}+1$ vertices of $N(I)$, which is impossible.
	Thus, every edge of $M$ joins a vertex of $N(I)$ to a vertex of $I$, and $M$ matches $N(I)$ into $I$.

	Let $v\in\core G\ii L(G)$. By \cref{asdklasdookokasdklasdookokasdklasdookok},
	we have $v\in\nucleus G$, and hence $v\in I$.
	If $v\notin D(L_{G})$, then every maximum matching of $L_{G}$ covers $v$.
	In particular, there exists $u\in N(I)$ such that $M(u)=v$.
	If $u\in\corona G$, then some set $S\in\Omega(G)$ contains $u$.
	Since $v\in\core G$, the same set $S$ also contains $v$, contradicting the fact that $uv\in E(G)$.
	Therefore, $u\notin\corona G$, and so $v\in M\left(N(I)-\corona G\right)$.
	This proves that
	\[
	\core G\ii L(G)\subseteq M\left(N(I)-\corona G\right)\u D(L_{G}).
	\]

	Now let $S\in\Omega(G)$. By \cref{larsonthm}, the set $S\ii L(G)$ belongs to $\Omega(L_{G})$.
	Since $M$ has $\a{N(I)}$ edges and its unmatched vertices lie in $I$, every set in $\Omega(L_{G})$
	contains exactly one endpoint of each edge of $M$ and every unmatched vertex of $M$.
	Therefore, if $x\in D(L_{G})$, then there exists a maximum matching $M_{x}$ of $L_{G}$ that misses $x$,
	and the previous observation applied to $M_{x}$ shows that $x\in S\ii L(G)$.
	As $S\in\Omega(G)$ was arbitrary, it follows that $D(L_{G})\subseteq\core G\ii L(G)$.

	Finally, let $v\in M\left(N(I)-\corona G\right)$, and choose $u\in N(I)-\corona G$ such that $M(u)=v$.
	For every set $S\in\Omega(G)$, the set $S\ii L(G)$ lies in $\Omega(L_{G})$ and therefore contains
	exactly one endpoint of each edge of $M$.
	Since $u\notin\corona G$, no maximum independent set of $G$ contains $u$, and hence every such set $S$
	must contain $v$.
	Therefore, $v\in\core G\ii L(G)$, and this proves the reverse inclusion.
\end{proof}

As a consequence of \cref{asdklasdookokasdklasdookokasdklasdookok}, Lemma \ref{asdklasdookok}, and \cref{asdklasdookokasdklasdookok}, we infer the following.

\begin{theorem}
	If $I$ is a maximum critical independent set of $G$, then
	\[
	\core G=\core{L_{G}^{c}}\u M\left(N(I)-\corona G\right)\u\ker G,
	\]
	\noindent for every maximum matching $M$ of $L_{G}$. 
\end{theorem}

\begin{theorem}[\label{aspijk123}\cite{levit2012critical}]
	$G$ is a König--Egerváry graph if and only if each of its maximum
	independent sets is critical.
\end{theorem}

From \cref{aspijk123} the following is obtained directly.

\begin{theorem}\label{asdijasodijasodij}
	If $G$ is a König--Egerváry graph, then 
	\[
	\mcritical G=\Omega(G),
	\]
	\noindent that is, $\nucleus G=\core G$. 
\end{theorem}

\begin{theorem}[\label{asdmasdkmaskdm}\cite{lmkcorecritical}]
	 The equality $d(G)=d(L_{G})$ holds for every graph $G$.
\end{theorem}

Having described $\core G$, we now turn to the family of maximum
critical independent sets.

\begin{theorem}\label{asiodkasodkaso}
	The equality
	\[
	\mcritical G=\left\{ S\in\Omega\left(L_{G}\right):E(S,L^{c}(G))=\emptyset\right\} 
	\] is valid for every graph $G$.
\end{theorem}

\begin{proof}
	Let $I\in\mcritical G$. By \cref{larsonthm}, we have $L(G)=I\u N(I)$ and $L_{G}$ is a König--Egerváry graph.
	By \cref{matchinglemma}, there exists a maximum matching of $G$ that matches $N(I)$ into $I$.
	As in the proof of \cref{asdklasdookok}, this implies that $\mu(L_{G})\ge \a{N(I)}$.
	Since $I$ is independent in $L_{G}$ and $\alpha(L_{G})+\mu(L_{G})=\a{L(G)}=\a I+\a{N(I)}$,
	we conclude that $I\in\Omega(L_{G})$. Moreover, note that $E(I,L^{c}(G))=\emptyset$.

	Conversely, let $I\in\Omega\left(L_{G}\right)$ such that $E(I,L^{c}(G))=\emptyset$.
	By \cref{asdijasodijasodij} we have $I\in\mcritical{L_{G}}$, and therefore
	$\a I-\a{N_{L_{G}}(I)}=d(L_{G})$.
	Note that $N_{L_{G}}(I)=N_{G}(I)$, since $E(I,L^{c}(G))=\emptyset$.
	In addition, by \cref{asdmasdkmaskdm}, we have $d(G)=d(L_{G})$.
	Therefore, $\a I-\a{N_{G}(I)}=d(G)$, and hence $I$ is a critical independent set of $G$.
	If $J\in\mcritical G$, then the first part of the proof shows that $J\in\Omega(L_{G})$.
	Thus $\a J=\alpha(L_{G})=\a I$, so $I$ is a maximum critical independent set of $G$.
\end{proof}

\begin{corollary}\label{diademlarson}
	For every graph $G$,
	\[
	\nucleus G=\I\left\{ S\in\Omega\left(L_{G}\right):E(S,L^{c}(G))=\emptyset\right\}
	\]
	and
	\[
	\diadem G=\U\left\{ S\in\Omega\left(L_{G}\right):E(S,L^{c}(G))=\emptyset\right\}\subseteq \corona G\ii L(G).
	\]
\end{corollary}

\begin{proof}
	The first equality is an immediate reformulation of \cref{asiodkasodkaso}. The second follows from the definition of $\diadem G$ and the same characterization of $\mcritical G$. If $I\in\mcritical G$, then by \cref{asiodkasodkaso} we have $I\in\Omega(L_G)$ and $E(I,L^{c}(G))=\emptyset$. Let $R\in\Omega(L_{G}^{c})$. Since $I\u R$ is independent and, by \cref{larsonthm},
	\[
	\a(I\u R)=\alpha(L_G)+\alpha(L_G^c)=\alpha(G),
	\]
	it follows that $I\u R\in\Omega(G)$. Hence $I\subseteq \corona G\ii L(G)$, and taking unions over all sets $I\in\mcritical G$ yields the desired inclusion.
\end{proof}

\begin{theorem}\label{asodpkaspdokapo}
	If $I$ is a maximum critical independent set of $G$, then
	\begin{align*}
		M\left(N(I)-\left(\corona G-\partial_{L}(G)\right)\right)\u\ker G\subseteq\nucleus G
	\end{align*}
	\noindent holds for every maximum matching $M$ of $L_{G}$. 
\end{theorem}

\begin{proof}
	Let $I\in\mcritical G$. By \cref{larsonthm}, we know that $L(G)=I\u N(I)$.
	Let $M$ be a maximum matching of $L_{G}$. As in the proof of \cref{asdklasdookok},
	$M$ matches $N(I)$ into $I$, and every maximum independent set of $L_{G}$ contains
	exactly one endpoint of each edge of $M$.
	By \cref{asiodkasodkaso},
	\[
	\nucleus G = \I\left\{ S\in\Omega\left(L_{G}\right):E(S,L^{c}(G))=\emptyset\right\}.
	\]
		Let $v\in M\left(N(I)-\left(\corona G-\partial_{L}(G)\right)\right)$.
		Choose $u\in N(I)-\left(\corona G-\partial_{L}(G)\right)$ such that $M(u)=v$.

	If $u\notin\corona G$, then \cref{asdklasdookok} yields
	$v\in\core G\ii L(G)\subseteq\nucleus G$.
	Suppose now that $u\in\partial_{L}(G)$.
	Let $S\in\Omega\left(L_{G}\right)$ satisfy $E(S,L^{c}(G))=\emptyset$.
	Since $u$ has a neighbor in $L^{c}(G)$, we must have $u\notin S$.
	Because $S$ contains exactly one endpoint of each edge of $M$, it follows that $v=M(u)\in S$.
	As $S$ was arbitrary, we conclude that $v\in\nucleus G$.

	Finally, since $\ker G\subseteq\nucleus G$, we obtain
	\[
	M\left(N(I)-\left(\corona G-\partial_{L}(G)\right)\right)\u\ker G\subseteq\nucleus G,
	\]
	which completes the proof.
\end{proof}

\begin{theorem}
	If $G$ is a König--Egerváry graph, then
	\[
	\nucleus G=M\left(N(I)-\corona G\right)\u\ker G.
	\]
\end{theorem}

\begin{proof}
	Let $I\in\mcritical G$. Since $G$ is a König--Egerváry graph,
	\cref{asdijasodijasodij} implies that $I\in\Omega(G)$.
	Hence $I\u N(I)=V(G)$, and by \cref{larsonthm} we obtain $L(G)=V(G)$, that is, $L_{G}=G$.
	Again by \cref{asdijasodijasodij}, we have $\nucleus G=\core G$.
		Therefore, \cref{asdklasdookok,asdokasdpokas} yield
		\[
		\begin{aligned}
		\nucleus G
		&=\core G\\
		&=M\left(N(I)-\corona G\right)\u D(L_{G})\\
		&=M\left(N(I)-\corona G\right)\u\ker G.
		\end{aligned}
		\]
\end{proof}

We now arrive at the key point of the argument.
The following theorem shows that, once the boundary obstruction disappears,
both the family of maximum critical independent sets and the corresponding invariants are completely determined by the
König--Egerváry part $L_G$.

\begin{theorem}\label{asdokasokd}
	Let $G$ be a graph such that $\corona G\ii\partial_{L}(G)=\emptyset$.
	Then
	\[
	\mcritical G=\left\{ T\ii L(G):T\in\Omega(G)\right\}.
	\]
	Consequently,
	\[
	\nucleus G=\core G\ii L(G)
	\qquad\textnormal{and}\qquad
	\diadem G=\corona G\ii L(G).
	\]
\end{theorem}

\begin{proof}
	By \cref{asiodkasodkaso},
	\[
	\mcritical G=\left\{ S\in\Omega\left(L_{G}\right):E(S,L^{c}(G))=\emptyset\right\}.
	\]
	We claim that
	\[
	\left\{ S\in\Omega\left(L_{G}\right):E(S,L^{c}(G))=\emptyset\right\}
	=
	\left\{ T\ii L(G):T\in\Omega(G)\right\}.
	\]
	Let $S\in\Omega\left(L_{G}\right)$ satisfy $E(S,L^{c}(G))=\emptyset$.
	If $R\in\Omega\left(L_{G}^{c}\right)$, then $S\u R$ is an independent set of cardinality
	$\alpha(L_{G})+\alpha(L_{G}^{c})=\alpha(G)$, and hence $S\u R\in\Omega(G)$.
	Thus $S=(S\u R)\ii L(G)$.
	Conversely, let $T\in\Omega(G)$.
	By \cref{larsonthm}, the set $T\ii L(G)$ belongs to $\Omega(L_{G})$.
	Moreover, $T\ii L(G)\subseteq\corona G\ii L(G)$.
	Since $\corona G\ii\partial_{L}(G)=\emptyset$, no vertex of $T\ii L(G)$ is adjacent to $L^{c}(G)$,
	and therefore $E(T\ii L(G),L^{c}(G))=\emptyset$.
	This proves the claim. Taking intersections and unions over the two equal families yields
	\[
	\nucleus G=\core G\ii L(G)
	\qquad\textnormal{and}\qquad
	\diadem G=\corona G\ii L(G),
	\]
	as required.
\end{proof}

Combining the previous results, we obtain the main characterization theorem.

\begin{theorem}\label{asdikjasoid}
	A graph $G$ satisfies $\core G=\nucleus G$ if
	and only if 
	\[
	\core{L_{G}^{c}}=\corona G\ii\partial_{L}(G)=\emptyset.
	\]
\end{theorem}

\begin{proof}
	Let $I\in\mcritical G$. By \cref{larsonthm}, $L(G)=I\u N(I)$.
	Since $\nucleus G\subseteq I\subseteq L(G)$ and $\core G=\nucleus G$, it follows that $\core G\subseteq L(G)$.
	By \cref{asdklasdookokasdklasdookokasdklasdookok},
	\[
	\core G=\core{L_{G}^{c}}\u\left(\core G\ii L(G)\right),
	\]
	and therefore $\core{L_{G}^{c}}=\emptyset$.

	Suppose now that there exists $u\in\corona G\ii\partial_{L}(G)$.
	Let $M$ be a maximum matching of $L_{G}$.
	As in the proof of \cref{asdklasdookok}, the matching $M$ matches $N(I)$ into $I$.
	Set $v=M(u)$.
	Since $u\in\corona G$, some maximum independent set $T$ of $G$ contains $u$.
	Because $uv\in E(G)$, we have $v\notin T$, and hence $v\notin\core G=\nucleus G$.

	On the other hand, let $J\in\mcritical G$.
	By \cref{asiodkasodkaso}, we have $J\in\Omega\left(L_{G}\right)$ and $E(J,L^{c}(G))=\emptyset$.
	Since $u\in\partial_{L}(G)$, it follows that $u\notin J$.
	As in the proof of \cref{asdklasdookok}, every maximum independent set of $L_{G}$
	contains exactly one endpoint of each edge of $M$.
	Therefore, $v=M(u)\in J$.
	Since $J$ was arbitrary, $v\in\nucleus G$, a contradiction.
	Hence $\corona G\ii\partial_{L}(G)=\emptyset$.
	
	\medskip
	
	Conversely, suppose that $\core{L_{G}^{c}}=\corona G\ii\partial_{L}(G)=\emptyset$.
	Then \cref{asdklasdookokasdklasdookokasdklasdookok} implies that $\core G=\core G\ii L(G)$.
	Hence, by \cref{asdokasokd},
	\[
	\core G=\core G\ii L(G)=\nucleus G,
	\]
	as claimed.
\end{proof}

\begin{corollary}\label{diademboundary}
	A graph $G$ satisfies
	\[
	\diadem G=\corona G\ii L(G)
	\]
	if and only if
	\[
	\corona G\ii \partial_{L}(G)=\emptyset.
	\]
\end{corollary}

\begin{proof}
	If $\corona G\ii \partial_{L}(G)=\emptyset$, then \cref{asdokasokd} yields $\diadem G=\corona G\ii L(G)$. Conversely, assume that $\diadem G=\corona G\ii L(G)$ and suppose that there exists $u\in \corona G\ii\partial_{L}(G)$. Since $u\in L(G)$, the assumed equality implies that $u\in\diadem G$, and therefore some set $J\in\mcritical G$ contains $u$. However, \cref{asiodkasodkaso} yields $E(J,L^{c}(G))=\emptyset$, which is impossible because $u\in\partial_{L}(G)$. Hence $\corona G\ii \partial_{L}(G)=\emptyset$.
\end{proof}

\begin{corollary}\label{coreeqnucleusdiadem}
	A graph $G$ satisfies $\core G=\nucleus G$ if and only if
	\[
	\core{L_{G}^{c}}=\emptyset
	\qquad\textnormal{and}\qquad
	\diadem G=\corona G\ii L(G).
	\]
\end{corollary}

\begin{proof}
	This follows immediately from \cref{asdikjasoid,diademboundary}.
\end{proof}

\cref{asdikjasoid} completes the solution of \cref{ioasdoiasi12}.
It reveals that the behavior of $\core G$, $\nucleus G$, and $\diadem G$ is completely governed by Larson's decomposition.

An almost-bipartite graph is a graph containing a unique odd cycle. In
an almost-bipartite non-König--Egerváry graph, the Larson decomposition
is known: $L_{G}$ is a bipartite graph, while
$L_{G}^{c}$ is an odd cycle \cite{lmkcorecritical}. Then $\core{L^c_G}=\emptyset$. Additionally, $L_{G}$
can be decomposed via the Gallai-Edmonds structure: $D(L_{G}),A(L_{G}),C(L_{G})$.
It is easy to see by \cref{ge} that 
\[
D(L_{G})\ii\partial_{L}(G)= A(L_G)\ii\corona G=\emptyset.
\]
Therefore, by \cref{asdikjasoid} the following holds.

\begin{theorem}
	 An almost bipartite non-König--Egerváry graph
$G$ satisfies $\core G=\nucleus G$ if and only if
\[
C(L_{G})\ii \corona G \ii\partial_{L}(G)=\emptyset.
\]
\end{theorem}

\section{Conclusions}

In this paper we solved the problem posed in \cite{jarden2019monotonic} of characterizing the graphs satisfying $\core G=\nucleus G$. The solution is expressed in terms of Larson's independence decomposition
\[
V(G)=L(G)\u L^{c}(G).
\]
Our main theorem shows that the equality $\core G=\nucleus G$ is governed by two independent obstructions: the presence of vertices in $\core{L_{G}^{c}}$, which comes from the $2$-bicritical side, and the presence of vertices of $\corona G$ on the boundary $\partial_{L}(G)$, which measures the interaction between the two parts.

The auxiliary results obtained along the way provide a more detailed description of the maximum critical structure of a graph. In particular, \cref{asiodkasodkaso} identifies the family $\mcritical G$ with the maximum independent sets of $L_G$ that avoid $L_G^c$, while \cref{asdokasokd,diademboundary} show that the boundary condition $\corona G\ii \partial_{L}(G)=\emptyset$ is exactly the condition under which the maximum critical sets are the traces on $L(G)$ of the maximum independent sets of $G$. Consequently,
\[
\nucleus G=\core G\ii L(G)
\qquad\textnormal{and}\qquad
\diadem G=\corona G\ii L(G)
\]
whenever no vertex of $\corona G$ lies on the Larson boundary. For König--Egerváry graphs this recovers the classical identities $\nucleus G=\core G$ and $\diadem G=\corona G$. It is worth noting that $\diadem{G}=\corona {G}$ if and only if $G$ is a König--Egerváry graph \cite{Levit2015OnKC,short2015some}.

The almost-bipartite case illustrates the usefulness of the characterization. Since $L_G^c$ is an odd cycle in that setting, the obstruction coming from $\core{L_{G}^{c}}$ disappears automatically, and the problem reduces to a condition on the Gallai--Edmonds decomposition of the König--Egerváry part $L_G$. This suggests that Larson's decomposition provides a robust framework for further investigations relating maximum independent sets and maximum critical independent sets.

\section{Open problems}\label{open}

The results of this paper suggest several natural directions for further research.

\begin{problem}
	Give an explicit structural description of $\diadem G$ for an arbitrary graph $G$ in terms of the Larson decomposition $V(G)=L(G)\u L^{c}(G)$ and the boundary set $\partial_{L}(G)$.
\end{problem}

\begin{problem}
	Determine the algorithmic complexity of deciding whether a graph $G$ satisfies $\core G=\nucleus G$. More generally, study the complexity of computing $\nucleus G$ and $\diadem G$ from a Larson decomposition.
\end{problem}

\begin{problem}
	Develop weighted analogues of the notions of core, nucleus, corona, and diadem, and characterize when the weighted core coincides with the weighted nucleus.
\end{problem}

\section*{Acknowledgments}

	This work was partially supported by Universidad Nacional de San Luis, grants PROICO 03-0723 and PROIPRO 03-2923, MATH AmSud, grant 22-MATH-02, Consejo Nacional de Investigaciones
	Cient\'ificas y T\'ecnicas grant PIP 11220220100068CO and Agencia I+D+I grants PICT 2020-00549 and PICT 2020-04064.

	\section*{Declaration of generative AI and AI-assisted technologies in the writing process}
	During the preparation of this work the authors used ChatGPT-3.5 in order to improve the grammar of several paragraphs of the text. After using this service, the authors reviewed and edited the content as needed and take full responsibility for the content of the publication.

\section*{Data availability}

Data sharing not applicable to this article as no datasets were generated or analyzed during the current study.

\section*{Declarations}

\noindent\textbf{Conflict of interest} \ The authors declare that they have no conflict of interest.

\bibliographystyle{apalike}

\bibliography{TAGcitasV2025_revised_bibfix}

\end{document}